\documentclass[12pt,a4paper]{amsart}
\usepackage{mathrsfs}
\usepackage{latexsym}
\usepackage{amsmath,amsthm,amscd,amsfonts,amssymb}
\usepackage{mathptmx,amssymb,color}

\headsep=1cm \numberwithin{equation}{section}
\newtheorem{theorem}{Theorem}[section]
\newtheorem{lemma}[theorem]{Lemma}

\newtheorem{corollary}[theorem]{Corollary}

\theoremstyle{definition}
\newtheorem{definition}[theorem]{Definition}

\newtheorem{remark}[theorem]{Remark}

\newcommand\mAss{\operatorname{mAss}}
\newcommand\Ass{\operatorname{Ass}}

\newcommand\Rad{\operatorname{Rad}}

\begin{document}

\title[Symbolic powers and generalized-parametric decomposition of monomial ideals]
{Symbolic powers and generalized-parametric decomposition of monomial ideals on regular sequences}%
\author[Adeleh Azari,  Simin Mollamahmoudi and Reza Naghipour]
{Adeleh Azari,  Simin Mollamahmoudi and Reza Naghipour$^*$\\\\\\
\vspace*{0.5cm}Dedicated to Professor Peter Schenzel}%
\address{Department of Mathematics, University of Tabriz,
Tabriz, Iran, and, School of Mathematics, Institute for Research in Fundamental
Sciences (IPM), P.O. Box: 19395-5746, Tehran, Iran.}%
\email{adeleh\_azari@yahoo.com (Adeleh Azari)}
\email{mahmoudi.simin@yahoo.com (Simin Mollamahmoudi)}

\email{naghipour@ipm.ir (Reza Naghipour)} \email  {naghipour@tabrizu.ac.ir (Reza Naghipour)}

\thanks{ 2010 {\it Mathematics Subject Classification}: 13B22.\\
This research was in part supported by a grant from IPM. \\
$^*$Corresponding author: e-mail: {\it naghipour@ipm.ir} (Reza Naghipour)}%

\subjclass{}%
\keywords{Associated prime, generalized-parametric decomposition, integral closure,  monomial ideal, normal ideal, regular sequence, squarefree.}%

\begin{abstract}
Let $R$ be a commutative Noetherian  ring and let ${\bf x} :=x_1,\ldots,x_d$ be a regular  $R$-sequence contained in the Jacobson radical of  $R$. An   ideal  $I$ of $R$
 is said to be a monomial ideal with respect to ${\bf x}$ if it is generated by a set of monomials $x_1^{e_1}\ldots x_d^{e_d}$. It is shown that, if ${\bf x}R$ is a prime  ideal of 
 $R$, then each monomial ideal $I$ has
 a canonical and unique decomposition   as an irredundant   finite intersection   of primary ideals of the form  $x^{e_1}_{\tau(1)}R+\dots+x^{e_s}_{\tau(s)}R$, where  $\tau$
  is a permutation of $\{1,\ldots,d\}$, $s\in\{1,\ldots,d\}$ and ${e_1},\ldots,{e_s}$ are the positive integers.  This generalizes and provides a short proof of the main results
  of \cite{HMRS, HRS}.  Also, we show that for every integer $k\geq1$, $I^{(k)}=I^k$, if and only if $\Ass_R R/{I^k }\subseteq \Ass_R R/{I}$, whenever $I$ is a  squarefree
   monomial ideal, where $I^{(k)}$ is the $k$th symbolic power of $I$. Moreover, in this circumstance it is shown that all powers of $I$ are integrally closed.
\end{abstract}
\maketitle
\section {Introduction}
Let $R$ be a commutative Noetherian ring, and let ${\bf x}
:=x_1,\ldots,x_d$ be a regular  $R$-sequence.  A {\it monomial} with
respect to ${\bf x}$ is a power product
$x_1^{e_1},\ldots,x_d^{e_d}$, where ${e_1},\ldots,{e_d}$ are
non-negative integers ( so monomial is  either  a non-unit  or  the
element 1), and a {\it monomial ideal} is a
 proper ideal generated by monomials.

Monomial ideals are important in  several areas of current research in commutative Noetherian rings, and they have been studied  in their own right in several papers (for example
  see \cite{EH, HMRS, KS}),   so many interesting  results are proved about such ideals.  In particular, recently Heinzer et al. (see \cite[Theorem 4.9]{HMRS}
  and  \cite [Theorems 4.1 and 4.10]{HRS})  showed that each monomial ideal $I$ with $\Rad(I)=\Rad({\bf x}R)$ has a unique
generalized-parametric decomposition. Their proof relies some defnitions and preliminary results; and several pages are devoted to a proof of that theorem. The first main
purpose of  the present paper is
to  generalize and  present a much short proof of  Heinzer et al. theorems, using somewhat more elementary methods than those used by them.  More precisely, as
 a main result of the Section 2, we shall show that:

\begin{theorem}\label{thm1}
Let  $I$  denote a  monomial ideal of $R$ with respect to a regular $R$-sequence ${\mathbf{x}}:=x_1,\ldots,x_d$ contained   in the Jacobson radical of $R$. Then $I$ has a
 canonical and unique decomposition as an irredundant finite intersection of  primary  ideals of the form $x^{e_1}_{\tau(1)}R+\dots+x^{e_s}_{\tau(s)}R$, whenever 
 the ideal ${\bf x}R$ is prime in $R$.  Here $\tau$ is a
  permutation of $\{1,\ldots,d\}$, $1\leq s\leq d$ and ${e_1},\ldots,{e_s}$ are positive integers.
\end{theorem}
The result in  Theorem   \ref{thm1} is proved in Theorem \ref{thm3}.  Pursuing this point of view further we establish  that a  monomial ideal is irreducible  if and only if it is
 a generalized-parameter ideal.

One of our tools for proving Theorem \ref{thm1} is the following:
\begin{lemma}\label{lem1}
Let  $I$  be a  monomial ideal of $R$ with respect to the regular  sequence ${\bf x} :=x_1,\ldots,x_d$, and let $\{u_1,\ldots,u_r\}$  be a minimal set of monomial
generators with respect to ${\bf x}$ of $I$.
Suppose that $u_1=vw$, where $v$ and $w$ are co-prime monomials   with respect to ${\bf x}$ and $v\neq1\neq w$. Then
$$ I=(vR+u_2R+\dots+u_rR)\cap(wR+u_2R+\dots +Ru_r).$$
\end{lemma}

Let $J$ be an ideal of $R$. Recalling that an element $x \in R$ is
said to be {\it integral over} $J$ if there exist an integer
$n\geq1$ and elements $c_i\in J^i$ for $i=1, \dots, n$ such that
$$x^n+c_1x^{n-1}+\dots+c_n=0.$$

The set of all elements that are integral over $J$ is called the {\it integral closure} of $J$, denoted by $\bar{J}$. It is well known that $\bar{J}$ is an ideal of $R$ and that
 $J \subseteq \bar{J}$. If $J=\bar{J}$, then $J$ is called {\it integrally closed}; and
we say that $J$ is {\it normal} if for every integer $n\geq1$, $J^n$ is  integrally closed.

For any   regular $R$-sequence ${\bf x} :=x_1,\ldots,x_d$ contained   in the Jacobson radical of $R$, let ${\bf x}R$ be a prime ideal of $R$. In the Section 3,  we will obtain some
results  about the symbolic powers of a  certain monomial ideal in $R$.  In this section,  among  other things, we prove the following theorem.
\begin{theorem}\label{thm2}
Let  ${\bf x} =x_1,\ldots,x_d$ be a regular $R$-sequence  contained
in the Jacobson radical of $R$, let ${\bf x}R$ be a prime ideal of
$R$ and let  $I$  be a  monomial ideal of $R$ with respect to ${\bf
x}$. Then, for any integer $k\geq1$, the $k$th   symbolic power
$I^{(k)}$ of $I$  is a monomial ideal, and if $I$ is  squarefree,
then $I^{(k)}=I^k$, if and only if
 $\Ass_R R/{I^k }\subseteq \Ass_R R/I$. Moreover, in this situation  $I$  is normal and $I^k=\bigcap\limits_{\mathfrak{p}\in \mAss_R R/I} \mathfrak{p}^k$.
\end{theorem}

Here, for a positive integer $n$, $I^{(n)}$ denotes the $n$th symbolic power of $I$,  which is defined  as the intersection of the primary components of $I^n$  corresponding to the minimal
 associated   primes of $I$.

Throughout this paper all rings are commutative and Noetherian, with  identity, unless otherwise specified. We shall use $R$ to denote
 such a ring and $I$ an ideal of $R$.    The {\it radical} of $I$, denoted by $\Rad(I),$
 is defined to be the set
 $\lbrace x\in I: x^n\in I\,\,  \text {for some}\,\,  n\in \mathbb{N}\rbrace.$  Moreover,  we denote by $\mAss_RR/I$ the set of minimal prime ideals of $\Ass_RR/I$.  We say that
 $x_1,\ldots,x_d$ form an $R$-{\it sequence} (of elements of $R$) precisely when $x_1R+\dots+x_dR\neq R$ and  for each $i=1, \dots, d$, the element $x_i$ is a non-zerodivisor on the $R$-module
 $R/x_1R+\dots+x_{i-1}R$.

  For any unexplained notation and terminology we refer the reader to \cite{BHe} or \cite{Na}.

\section{Generalized-Parametric Decomposition of Monomial Ideal}
Let $R$ be a  Noetherian ring, let ${\bf x} =x_1,\ldots,x_d$ be a regular  $R$-sequence contained   in the Jacobson radical of $R$. In this section we show that that if $I$ is a monomial ideal
 of $R$ with respect to ${\bf x}$ then $I$ has a unique generalized-parametric decomposition of primary ideals. The main result of this section is Theorem 2.3, which extends and provides a
 short   proof of the main results of \cite{HMRS, HRS}. We begin with:
\begin{definition}\label{def1}
Let $R$ be a  Noetherian ring and let ${\bf x} :=x_1,\ldots,x_d$ be a regular  $R$-sequence of elements of $R$.  Then
\begin{enumerate}
  \item[(i)] An element $m$ of $R$ is called {\it monomial with respect to} ${\bf x}$ if there exist non-negative integers ${e_1},\ldots,{e_d}$ such that $m= x^{e_1}_1 \ldots  x^{e_d}_d$.
  In view of   \cite[Theorem 16.2]{Ma} it is easy to see that ${e_1},\ldots,{e_d}$ are determined uniquely by $m$.
  \item[(ii)]  Suppose  that $\mathscr{M}$ denotes the set of all monomials of $R$ with respect to ${\bf x}$. An ideal $I$ of $R$ is called {\it monomial with respect to} ${\bf x}$ if
   it is generated by elements of $\mathscr{M}$. In particular, the zero  ideal and $R$ itself are monomial ideals.
  \item[(iii)] Let $u=x^{e_1}_1\ldots x^{e_d}_{d}$ and $v=x^{t_1}_1\ldots x^{t_d}_{d}$ be two  monomials with respect to ${\bf x}$. For all $i\in\{1,\ldots,d\}$, we set $k_i=\min\{e_i,t_i\}$
   and $s_i=\max\{e_i,t_i\}$. Then, we define
      \[ \verb"gcd"(u,v)= x^{k_1}_1\ldots x^{k_d}_{d},\quad \verb"lcm"(u,v)= x^{s_1}_1\ldots x^{s_d}_{d},
      \]
      the {\it greatest  common divisor} resp. the {\it least common multiple} of $u$ and $v$. We say that $u$ and $v$ are co-prime if $ \verb"gcd"(u,v)=1$.
\item[(iv)] Suppose that $s$ is an integer such that $1\leq s\leq d,$  let $\tau$ be a  permutation  of $\{1,\ldots,d\}$, and let ${e_1},\ldots,{e_d}$  be positive integers. Then the monomial
 ideal generated by $x^{e_1}_{\tau(1)},\ldots, x^{e_s}_{\tau(s)}$ is called a  {\it generalized-parametric ideal}.
\item[(v)] A monomial ideal $I$ of $R$ with respect to  ${\bf x}$   is called {\it reducible} if there exist  two monomial  ideals $I_1$, $I_2$ of $R$ with respect to  ${\bf x}$ such that
$I=I_1\cap I_2$ and $I\neq I_1, I_2$.
    It  is called {\it  irreducible} if it   is not reducible.
\end{enumerate}
\end{definition}
Throughout this section, $R$ denotes a  commutative Noetherian ring and ${\bf x} :=x_1,\ldots,x_d$ is a regular  $R$-sequence contained   in the Jacobson radical of $R$, unless otherwise
specified.

Before we state the main theorem of this section, we need the following lemma which plays a key role in the proof of that theorem.
\begin{lemma}\label{lem2}
Let  $I$  be a non-zero monomial ideal of $R$ with respect to  ${\bf x}$. Let $\{u_1,\ldots,u_r\}$  be a minimal set of   generators   of $I$.
Let $u_1=v_1w_1$ where $v_1$ and $w_1$ are co-prime monomials, i.e.,  $v_1\neq1\neq w_1$ and $\verb"gcd"(w_1,v_1)=1$. Then
$$ I=(v_1R+u_2R+\cdots +u_rR)\cap(w_1R+u_2R+\cdots +u_rR).$$
\end{lemma}
\begin{proof}
Let
\begin{center}
$I_1:=v_1R+u_2R+\cdots +u_rR$ and $I_2:=w_1R+u_2R+\cdots +u_rR$.
\end{center}
 Then, it is clear that  $I\subseteq I_1\cap I_2$. On the other hand , in view of \cite[Proposition 1]{KS}, we have
\begin{eqnarray*}
I_1\cap I_2 &=& \verb"lcm"(v_1,w_1)R+\verb"lcm"(v_1,u_2)R+\cdots+\verb"lcm"(v_1,u_r)R\\
&+& \verb"lcm"(u_2,w_1)R+\verb"lcm"(u_2,u_2)R+\cdots+\verb"lcm"(u_2,u_r)R\\
&\vdots&\\
&+& \verb"lcm"(u_r,w_1)R+\verb"lcm"(u_r,u_2)R+\cdots+\verb"lcm"(u_r,u_r)R.
\end{eqnarray*}
Now, since $\verb"gcd"(w_1,v_1)=1$, it follows that $\verb"lcm"(w_1,v_1)=v_1w_1=u_1$, and so
\[
I_1\cap I_2\subseteq u_1R+u_2R+\cdots+u_rR=I.
\]
Therefore  $I=I_1\cap I_2$, as required.
\end{proof}

We are now ready to  prove the main theorem of this section, which generalizes and provides a short proof of the main results of  Heinzer et al. (see \cite[Theorem 4.9]{HMRS}
and  \cite[Theorems 4.1 and 4.10]{HRS}).
\begin{theorem}\label{thm3}
Let  $I$  be a non-zero monomial ideal of $R$ with respect to  ${\bf x}$. Then $I$ is an irredundant finite intersection generalized-parametric ideals, that is,
 $I=\bigcap\limits_{i=1}^m \mathfrak{q}_i$, where each $\mathfrak{q}_i$ is of the form $x^{e_1}_{i_1}R+\cdots +x^{e_k}_{i_k}R$. Moreover, such an irredundant presentation
 (up to the order of the factors) is unique; and if  ${\bf x}R$ is a prime ideal of $R$, then the ideal $\mathfrak{q}_i$ is $x_{i_1}R+\cdots+x_{i_k}R$-primary.
\end{theorem}
\begin{proof}
Let  $I$  be a non-zero monomial ideal of $R$ with respect to  ${\bf x}$, and let $\{u_1,\ldots,u_r\}$  be a minimal system of   generators   of $I$.
If every $u_i$ has pure power, then being nothing to prove. So suppose that some $u_i$ is not a pure power, say $u_1$. Then we can write $u_1=vw$, where $v$ and $w$ are
 monomials with $\verb"gcd"(v,w)=1$. Hence, in view of Lemma \ref{lem2}, we have $I=I_1\cap I_2$, where
\begin{center}
$I_1:=vR+u_2R+\cdots +u_rR$ and $I_2:=wR+u_2R+\cdots +u_rR$.
\end{center}

 Now, if $\{v,u_2,\ldots,u_r\}$ or $\{w,u_2,\ldots,u_r\}$ contains an element which is not a pure power, we proceed as before and obtain   a finite number of steps a presentation of $I$
  as an intersection of monomial ideals generated by pure powers. That is $I$ is a finite intersection of   generalized-parameter ideals. Now, by omitting those ideals which contains the
  intersection of the others we end up with an irredundant intersection of generalized-parameter ideals.

Next, let $I=\mathfrak{q}_1\cap\ldots\cap \mathfrak{q}_r$ and $I=\mathfrak{q}'_1\cap\ldots\cap \mathfrak{q}'_s$ be two irredundant decomposition of $I$ as finite intersection
 of generated-parameter ideals.
 We need to show that $r=s$ and that $$\{\mathfrak{q}_1,\ldots,\mathfrak{q}_r\}=\{\mathfrak{q}'_1\,\ldots,\mathfrak{q}'_s\}.$$
In order to do so, we show that for each $i=1,\ldots,r$, there exists $j=1,\ldots,s$ such that $\mathfrak{q}'_j\subseteq \mathfrak{q}_i$, and by symmetry we then  also have
that for each $k=1,\ldots,s$, there exists $l=1,\ldots,r$ such that $\mathfrak{q}_l\subseteq \mathfrak{q}'_k$.  Now, let $i\in\{1,\ldots,r\}$, and we suppose that
 $\mathfrak{q}'_j\nsubseteq \mathfrak{q}_i$ for all $j=1,\ldots,s$, and look for a contradiction. Then, we may assume that $$\mathfrak{q}_i=x_1^{e_1}R+\cdots +x_k^{e_k}R,$$
 and  for every $j=1,\ldots,s$, there exists $x_{l_j}^{b_j}\in \mathfrak{q}'_j\setminus \mathfrak{q}_i$. Hence, it follows that either
$l_j\notin\{1,\ldots,k\}$ or $b_j<e_{l_j}$. Now, set
$$u=\verb"lcm"(x_{l_1}^{b_1},\ldots,x_{l_s}^{b_s}).$$
Then we have $u\in\bigcap\limits_{j=1}^s \mathfrak{q}'_j=I$, and so $u\in\bigcap\limits_{j=1}^r \mathfrak{q}_j.$ In particular
$$u\in \mathfrak{q}_i=x_{i_1}^{e_1}R+\cdots +x_{i_k}^{e_k}R.$$ Consequently, in view of    \cite[Corollary 3]{KS}, there exists $i\in\{1,\ldots,k\}$
such that $x_i^{e_i}|u$, and so $u\in x_i^{e_i}R$. Hence, it follows from   \cite[Remark 1]{KS} that there is $1\leq t\leq s$, $b_i\geq e_i$, which is a contradiction.

Finally, let  ${\bf x}R$ be a prime ideal of $R$. Then, since each
$\mathfrak{q}_i$  is of the form $x_{i_1}^{e_1}R+\cdots +x_{i_k}^{e_k}R$, it
follows from   \cite[Theorem 3.4]{Se} that $\mathfrak{q}_i$ is a
$x_{i_1}R+\cdots+x_{i_k}R$-primary ideal. This completes the proof.
\end{proof}

\begin{remark}\label{rem1}
Let ${\bf x} =x_1,\ldots,x_d$ be a regular $R$-sequence contained   in the Jacobson radical of $R$  and let  ${\bf x}R$ be a prime ideal of $R$. Let $I$ be a non-zero monomial ideal of $R$
with respect to ${\bf x}$. Then by virtue of  Theorem  \ref{thm3}, every element of $\Ass_R R/{I}$ is a monomial ideal. In particular every element $\mAss_R R/{I}$ is a monomial ideal.
Furthermore, since $\Rad(I)=\bigcap\limits_{\mathfrak{p}\in \mAss_R R/I} \mathfrak{p}$,  it follows that from  \cite[Proposition 1]{KS} that the ideal $\Rad(I)$ is monomial, too.

Next,   as $I^n$  for every integer $n\geq1$, is a monomial ideal, it follows from definition that $I^{(n)}$ is also a monomial ideal.\\

The first application of Theorem 2.3 shows that a  monomial ideal is irreducible  if and only if it is  a generalized-parameter ideal.
\end{remark}
\begin{corollary}\label{cor1}
Let ${\bf x} =x_1,\ldots,x_d$ be a regular $R$-sequence contained   in the Jacobson radical of $R$. Then a monomial ideal with respect to  ${\bf x}$ is irreducible  if and only if it is
 a generalized-parameter ideal.
\end{corollary}
\begin{proof}
Let  $I$  be a  monomial ideal of $R$ with respect to   ${\bf x}$. If $I$ is irreducible and $\{u_1,\ldots,u_r\}$ is a monomial system of generators of $I$ such that some $u_i$ is not a pure
 power, say $u_1$, then   we can write $u_1=vw$, where $u$ and $w$ are coprime monomials and $v\neq1\neq w$. Then
$$ I=(vR+u_2R+\cdots +u_rR)\cap(wR+u_2R+\cdots +u_rR),$$
which is a contradiction.

Conversely, let $I$ be a generalized-parameter ideal, and suppose that $I$ is not irreducible. Then there exist integer $e_i\geq1$ such that $I=x_{i_1}^{e_1}R+\cdots+x_{i_k}^{e_k}R$, and that there are
 two monomial ideals $J$ and $K$ properly containing $I$ such that $I=J\cap K$. In view of the Theorem \ref{thm3}, we have $J=\bigcap\limits_{i=1}^r \frak q_i$ and
 $K=\bigcap\limits_{j=1}^s \frak q'_j$, where $\mathfrak{q}_i$ and $\mathfrak{q}'_i$  are generalized-parameter ideals. Hence
 $$I=(\bigcap\limits_{i=1}^r \mathfrak{q}_i)\cap(\bigcap\limits_{i=1}^s \mathfrak{q}'_i).$$ By omitting suitable ideals in the intersection on the right-hand side, we derive an irredundant
  of $I$.  Now the  uniqueness statement in Theorem \ref{thm3}, then implies that $I=\frak q_i$ or $I=\frak q'_j$, for some $i$ or $j$, which is a contradiction.
\end{proof}

An important question is when the ordinary and symbolic powers of an ideal in a commutative Noetherian ring are equal. This  question has been studied by several authors;
see for instance \cite{BH, ELS, Ho, HKV,  Mo}, and
has led to some interesting results.  The next result shows that the  ordinary and symbolic powers of a monomial prime ideal are equal.
\begin{corollary}\label{cor2}
Let ${\bf x} =x_1,\ldots,x_d$ be a regular $R$-sequence contained   in the Jacobson radical of $R$. Let $\mathfrak{p}$ be a monomial prime ideal of $R$ with respect to  ${\bf x}$.
Then  for all integer $n\geq1$, $\mathfrak{p}^{(n)}=\mathfrak{p}^n$.
\end{corollary}
\begin{proof}
Since $\mathfrak{p}$ is a monomial prime ideal, it follows that $\mathfrak{p}$ is a   generalized-parameter ideal. Hence we can assume that $\mathfrak{p}=x_{i_1}^{e_1}R+\cdots+x_{i_s}^{e_s}R$,
where $1\leq s\leq d$ and $e_i\in\mathbb{N}$. As $\mathfrak{p}$ is prime, it yields that $x_{i_1},\ldots,x_{i_s}\in\mathfrak{p}$, and so $\mathfrak{p}=x_{i_1}R+\cdots+x_{i_s}R$. Therefore, since $\mathfrak{p}$ is generated by a regular  $R$-sequence, it follows from   \cite[Theorem 125 and Exercise 13]{Ka} that $\Ass_R R/{\mathfrak{p}^n }\subseteq \Ass_R R/\mathfrak{p}$ for any $n\in\mathbb{N}$.  Consequently, $\Ass_R R/{\mathfrak{p}^n }=\{\mathfrak{p}\}$, for every $n\in\mathbb{N}$, and so $\mathfrak{p}^n$ is a $\mathfrak{p}$-primary ideal. Hence $\mathfrak{p}^{(n)}=\mathfrak{p}^n$, as required.
\end{proof}
\section{Symbolic Power of Squarefree  Monomial Ideals and Normality}
Recall that ${\bf x} =x_1,\ldots,x_d$ is a regular $R$-sequence
contained   in the Jacobson radical of $R$. In this section we study
some properties of the squarefree monomial ideals
 with respect to ${\bf x}$. Specifically, we show that if $I$ is a squarefree monomial   ideal, then $\Ass_RR/I^n \subseteq \Ass_RR/I$ if and only if  $I^{(n)}=I^n$, for every integer $n\geq1$.  Also,   it is shown that in this case all powesr of $I$ are integrally closed. We begin with the following definition.
\begin{definition}\label{def2}
A monomial $m=x_1^{e_1}\ldots x_d^{e_d}$  with respect to ${\bf x}$ is called {\it squarefree} if the all $e_i$ are 0 or 1. A monomial   ideal $I$ is called a {\it squarefree monomial ideal}
 if $I$ is generated by squarefree monomials. If $m=x_1^{e_1}\ldots x_d^{e_d}$ is a monomial with respect to ${\bf x}$, then the {\it support} of $m$, denoted by ${\rm supp}(m)$,
 is defined to be the set $\{j|j\in\{1,\ldots,d\}\, \textrm{and}\,\,  e_j\neq0\}$. Also the {\it radical}  of $m$, denoted by ${\rm rad}(m)$,
  is defined as ${\rm rad}(m):=\prod\limits_{j\in {\rm supp}(m)} x_j$. It is clear that if $m\in I$, then $({\rm rad}(m))^t\in I$, for some integer $t\geq1$.
  Also, it is easy to see that $m={\rm rad}(m)$ if and only if $m$ is a squarefree  monomial.
\end{definition}
In the next lemma, we can compute explicit the radical of a monomial
ideal.

\begin{lemma}\label{lem3}
Let  $I$  be a  monomial ideal of $R$ with respect to  ${\bf x}$, and let  $\{u_1,\ldots,u_r\}$  be a minimal set of  generators  of $I$. Then the set $\{{\rm rad}(u_i)|1\leq i\leq r\}$
is a set of generators of $\Rad(I)$.
\end{lemma}
\begin{proof}
Let $I=u_1R+\dots+u_rR$.  We  show that $\Rad(I)={\rm rad}(u_1)R+\dots +{\rm rad}(u_r)R$. To do this, as ${\rm rad}(u_j)\in \Rad(I)$ for every $1\leq j\leq r$, it follows that
$${\rm rad}(u_1)R+\dots+{\rm rad}(u_r)R\subseteq Rad(I).$$

Now in order to show the opposite inclusion, since $\Rad(I)$ is a monomial ideal with respect to $\bf x$ (see Remark \ref{rem1}), it is enough for us to show that for each monomial
$m\in \Rad(I)$ there exist monomial $m'$ and $1\leq i\leq r$ such that $m=m'{\rm rad}(u_i)$. To this end it follows from $m\in \Rad(I)$ that $m^l\in I$ for some integer $l\geq1$.
 Hence, in view of   \cite[Corollary 3]{KS}, there exists  a monomial $m'$ such that $m^l=m'u_j$ for some $1\leq j\leq r$. Now it is easy to see that this yields the desired conclusion.
\end{proof}
\begin{corollary}\label{cor3}
Let  $I$  be a  monomial ideal with respect to  ${\bf x}$. Then $\Rad(I)=I$ if and only if, $I$ is a squarefree monomial ideal with respect to ${\bf x}$.
In particular, every  monomial  prime ideal is squarefree.
\end{corollary}
\begin{proof}
The assertion readily  follows from Lemma \ref{lem3}.
\end{proof}
\begin{corollary}\label{cor4}
Let the ideal ${\bf x}R$ be a prime ideal of $R$. Then every squarefree monomial ideal is an intersection of monomial prime ideals.
\end{corollary}
\begin{proof}
Let $I$ be a squarefree monomial ideal with respect to ${\bf x}$, and let ${\bf x}R$ be a prime of $R$. Then in view of Theorem \ref{thm3},  we have $I=\bigcap\limits_{i=1}^m \mathfrak{q}_i$,
 where every $\mathfrak{q}_i$ is a  primary generalized-parameter ideal. Hence, for all $1\leq i\leq m,$ there exist positive integers $e_{i_1},\ldots,e_{i_k}$ such that
$$\mathfrak{q}_i=x_{i_1}^{e_{i_1}}R+\dots +x_{i_k}^{e_{i_k}}R.$$
 Since  $\Rad(\mathfrak{q}_i)=x_{i_1}R+\dots +x_{i_k}R$ is a monomial prime ideal, the desired conclusion follows from Corollary \ref{cor3}.
\end{proof}
\begin{corollary}\label{cor5}
Let  $I$  be a squarefree monomial  ideal with respect to  ${\bf x}$, and let ${\bf x}R$ be a  prime ideal of $R$. Then $I=\bigcap\limits_{\mathfrak{p}\in \mAss_R R/I} \mathfrak{p}$.
\end{corollary}
\begin{proof}
In view of Corollary \ref{cor4}, there exist monomial prime ideals $\mathfrak{p}_1,\ldots,\mathfrak{p}_r$ such that $I=\bigcap\limits_{i=1}^r \mathfrak{p}_i$. Now, it is easy to see that
 $\mathfrak{p}_i\in \mAss_R R/I$, for every $i=1,\ldots,r$. This completes the proof.
\end{proof}

The next lemma is almost certainly known, but we could not find a
reference for it, so it is explicit stated and proved here, since it
is needed for extending Corollary 2.6
 to an arbitrary squarefree monomial ideal; and also this used in the proof of the  main result of this section.
\begin{lemma}\label{prop1}
Let $\mathfrak{p}_1,\ldots,\mathfrak{p}_t$ be prime ideals of  $R$. Then for any integer $n\geq1$ we have
$$(\bigcap_{i=1}^t \frak{p}_i)^{(n)}=\bigcap^t_{i=1} \frak{p}_i^{(n)}.$$
\end{lemma}
\begin{proof}
Let $J:=\bigcap_{i=1}^t\mathfrak{p}_i$ and let $x\in J^{(n)}$.  Then $sx\in {J^n}$  for some $s\in R\setminus\bigcup\limits_{i=1}^t\mathfrak{p}_i$. Hence $sx\in{\mathfrak{p}_i^n}$ for all
 $1\leq i\leq t$. Since $s\in R\setminus\mathfrak{p}_i$, for every $i=1,\ldots, t$ , it yields that $x\in\mathfrak{p}_i^{(n)}$, and so  $x\in\bigcap\limits_{i=1}^t\mathfrak{p}_i^{(n)}$.
  Consequently, $J^{(n)}\subseteq \bigcap\limits_{i=1}^t\mathfrak{p}_i^{(n)}$. Now, since $$\Ass_RR/J^{(n)}=\{\mathfrak{p}_1,\ldots,\mathfrak{p}_t\}$$ and for all $j=1,\ldots, t$,
\begin{center}
$(\bigcap_{i=1}^t\mathfrak{p}_i^{(n)}/J^{(n)})_{\mathfrak{p}_j}=0,$
\end{center}
it follows  that $\bigcap\limits_{i=1}^t\mathfrak{p}_i^{(n)}=J^{(n)}$, as required.
\end{proof}

\begin{corollary}\label{cor6}
Let  $I$  be a squarefree monomial ideal with respect to  ${\bf x}$, and suppose that ${\bf x}R$ is a  prime ideal of $R$. Then for any integer $k\geq1$,
we have $$I^{(k)}=\bigcap\limits_{\mathfrak{p}\in \mAss_R R/I} \mathfrak{p}^k.$$
\end{corollary}
\begin{proof}
In view of  Corollary \ref{cor5}, we have
$$I=\bigcap_{\mathfrak{p}\in \mAss_R R/I} \mathfrak{p}.$$
Hence, it follows from Lemma \ref{prop1} that
$$I^{(k)}=\bigcap\limits_{\mathfrak{p}\in \mAss_R R/I} \mathfrak{p}^{(k)}.$$
Now, as $\mathfrak{p}^{(k)}=\mathfrak{p}^k$ by Corollary \ref{cor2}, we deduce that $I^{(k)}=\bigcap\limits_{\mathfrak{p}\in \mAss_R R/I} \mathfrak{p}^k$, as required.
\end{proof}

We are now ready to state and prove the following main result of this section which is an extension of  Corollary 2.6 for an arbitrary squarefree monomial ideal.
\begin{theorem}\label{thm4}
Let  $I$  be a  squarefree monomial ideal, and assume that the ideal ${\bf x}R$  of $R$ is prime. Then the following conditions are equivalent:
\begin{enumerate}
  \item [(i)] $I^{(k)}=I^k$ for all integers $k\geq1$.
  \item [(ii)] $\Ass_R R/I^k\subseteq \mAss_R R/I$ for all integers $k\geq1$.
\end{enumerate}
Moreover, if the equivalent   conditions hold, then $I$ is a normal ideal.
\end{theorem}
\begin{proof}
The implication  (i)$\Longrightarrow$(ii) follows from the definition of $I^{(k)}$.

 In order to show (ii)$\Longrightarrow$(i),  since, $I$ is squarefree, it follows from Corollary \ref{cor5} that
$I=\bigcap\limits_{\mathfrak{p}\in \Ass_R R/I} \mathfrak{p}$.  Hence,
$$\Ass_R R/I=\mAss_R R/I.$$
Therefore $\Ass_R R/I=\mAss_R R/I^k$, and so it follows from $\Ass_R R/I^k\subseteq\Ass_R R/I$ that
 $$\Ass_R R/I^k=\mAss_R R/I^k.$$
 Let $I^k=\bigcap\limits_{i=1}^m \mathfrak{q}_i$
 be an irredundant primary decomposition of $I^k$ with $\mathfrak{q}_i$ is $\mathfrak{p}_i$-primary. Then $$\Ass_R R/I^k=\{\mathfrak{p}_1,\ldots,\mathfrak{p}_m\}.$$
 Now, because of
 \begin{center}
 $I^{(k)}=\bigcap\limits_{\mathfrak{p}_i\in \mAss_R R/I^k} \mathfrak{q}_i$\,\,\,  and \,\,\,  $\Ass_R R/I^k=\mAss_R R/I^k$,
 \end{center}
it follows that $I^{(k)}=I^k$. This proves the equivalence of (i) and (ii).

Finally, in order to complete the proof, we have to show that for each $k$, the ideal $I^k$ is integrally closed, i.e., $\overline{I^{k}}=I^k$.
To do this, in view of   \cite[Proposition 4]{KS}, it is enough for us to show that for a monomial $m\in\mathscr{M}$ in which $m^l\in I^{kl}$ for some integer $l\geq1$, we have $m\in I^k$;
note   that by virtue of   \cite[Theorem 1]{KS}  the ideal $\overline{I^{k}}$ is monomial. Now, since by assumption ${I^{(j)}}=I^j$ for all integer $j\geq1$, and according to
 Corollary \ref{cor6}, $$I^{(j)}=\bigcap\limits_{\mathfrak{p}\in \mAss_R R/I} \mathfrak{p}^j,$$   it  suffices to prove that whenever
 $$m^{l}\in\bigcap\limits_{\mathfrak{p}\in \mAss_R R/I} \mathfrak{p}^{kl},$$ for some integer $l\geq1$, we have  $m\in\bigcap\limits_{\mathfrak{p}\in \mAss_R R/I} \mathfrak{p}^{k}$.

 To do this, let $m=x_1^{e_1}\ldots x_d^{e_d}$.  Then $m^l=x_1^{e_1l}\ldots x_d^{e_dl}$.
  Now, it easily follows from $m^{l}\in\bigcap\limits_{\mathfrak{p}\in \mAss_R R/I} \mathfrak{p}^{kl}$ that $e_il\geq kl$ for all $i=1,\ldots,d$
  in which $x_i\in\mathfrak{p}$ and all $\mathfrak{p}\in \mAss_R R/I$. This then implies that $e_i\geq k$ for all $i=1,\ldots,d$ for which $x_i\in\mathfrak{p}$
  and all $\mathfrak{p}\in \mAss_R R/I$, which yields the desired conclusion.
\end{proof}

\begin{center}
{\bf Acknowledgments}
\end{center}
The authors would like to thank Professor Hossein Zakeri for his reading of the first draft and valuable suggestions.
We also would like to thank the Institute for Research in Fundamental Sciences (IPM), for the financial support.

\end{document}